\documentclass[11pt,reqno]{amsart}

\usepackage[T1]{fontenc}
\usepackage[utf8]{inputenc}
\usepackage{lmodern}
\usepackage{microtype}
\usepackage{amsmath,amssymb,mathtools}
\usepackage{amsthm}
\usepackage{aliascnt}
\usepackage{enumitem}
\usepackage{xcolor}
\usepackage{xurl}
\usepackage{hyperref}
\usepackage[nameinlink,capitalize,noabbrev]{cleveref}

\hypersetup{
  colorlinks=true,
  linkcolor=blue!55!black,
  citecolor=green!35!black,
  urlcolor=blue!55!black,
  pdftitle={Uniform Comparison of Hyperbolic Ball Volumes on the Universal Cover},
  pdfauthor={Heng Zhang}
}

\newtheorem{theorem}{Theorem}[section]
\newaliascnt{proposition}{theorem}
\newtheorem{proposition}[proposition]{Proposition}
\newaliascnt{conjecture}{theorem}
\newtheorem{conjecture}[conjecture]{Conjecture}
\aliascntresetthe{conjecture}
\aliascntresetthe{proposition}
\newaliascnt{lemma}{theorem}
\newtheorem{lemma}[lemma]{Lemma}
\aliascntresetthe{lemma}
\newaliascnt{corollary}{theorem}

\aliascntresetthe{corollary}
\newaliascnt{claim}{theorem}

\aliascntresetthe{claim}
\theoremstyle{definition}
\newaliascnt{definition}{theorem}

\aliascntresetthe{definition}
\newaliascnt{convention}{theorem}
\newtheorem{convention}[convention]{Convention}
\aliascntresetthe{convention}
\theoremstyle{remark}
\newaliascnt{remark}{theorem}
\newtheorem{remark}[remark]{Remark}
\aliascntresetthe{remark}

\DeclareMathOperator{\Vol}{Vol}
\DeclareMathOperator{\Area}{Area}
\DeclareMathOperator{\diam}{diam}
\DeclareMathOperator{\sys}{sys}
\DeclareMathOperator{\Lip}{Lip}
\DeclareMathOperator{\dist}{dist}

\newcommand{\R}{\mathbb{R}}

\newcommand{\Hh}{\mathbb{H}}
\newcommand{\cN}{\mathcal{N}}
\newcommand{\cP}{\mathcal{P}}

\newcommand{\norm}[1]{\left\lVert #1\right\rVert}
\newcommand{\simp}[1]{\left\lVert #1\right\rVert_{\Delta}}
\newcommand{\abs}[1]{\left\lvert #1\right\rvert}
\newcommand{\set}[1]{\left\{#1\right\}}
\newcommand{\brak}[1]{\left[#1\right]}
\newcommand{\paren}[1]{\left(#1\right)}
\newcommand{\eps}{\varepsilon}
\newcommand{\whU}{\widehat U}
\newcommand{\whS}{\widehat S}

\title[Comparison Hyperbolic Ball on the Universal Cover]{Uniform Comparison of Hyperbolic Ball Volumes on the Universal Cover}
\author{Heng Zhang}
\address{School of Mathematical Sciences, University of Science and Technology of China, Hefei, China}
\email{hengz@mail.ustc.edu.cn}
\date{July 2026}
\subjclass[2020]{53C23, 57R19, 20F69}
\keywords{simplicial volume, universal cover, volume growth, systole, separating filtration, padded partition}

\begin{document}

\begin{abstract}
Let $\simp{M}$ denote the simplicial volume of $M$,
$V_r(X,h)=\sup_{x\in X}\Vol_h(B_h(x,r))$, and
$\Hh^n$ denotes hyperbolic $n$-space. We prove that, if a closed oriented $n$-manifold $M$ admits a hyperbolic metric, then there is a dimensional constant $\delta_n>0$ such that every Riemannian metric $g$ on $M$ with
$$
  \frac{\Vol_g(M)}{\simp{M}}<\delta_n
$$
satisfies
$$
  V_r(\widetilde M,\widetilde g)\ge V_r(\Hh^n)
  \quad\text{for every }r\ge 1.
$$
\end{abstract}

\maketitle

\section{Introduction}

\subsection{Background}

For a Riemannian manifold $(X,h)$ and a radius $r>0$, set
$$
  V_r(X,h)=\sup_{x\in X}\Vol_h\bigl(B_h(x,r)\bigr),
$$
and let $\sys(X,h)$ denote the infimum of the lengths of non-null-homotopic
loops, with the convention $\sys(X,h)=\infty$ when $X$ is simply connected.
We write $V_r(X)$ and $\sys(X)$ when the metric is understood.  If $M$ is a
closed oriented connected manifold, we denote its Gromov simplicial volume by
$\simp{M}$.  When $g$ is a Riemannian metric on $M$, we write
$(\widetilde M,\widetilde g)$ for the universal cover equipped with the lifted
metric.

The problem considered here is to find geometric hypotheses ensuring that, when
$M$ admits a hyperbolic metric, balls in $(\widetilde M,\widetilde g)$ are at
least as large as balls of the same radius in hyperbolic space.  

Sabourau proved an all-scale comparison under an absolute small volume
assumption \cite{Sabourau}.  More
precisely, 

\begin{theorem}
For every $n\ge2$ there exists a constant $\delta_n>0$ such that, if
$M$ is a closed oriented connected $n$-manifold admitting a hyperbolic metric
and $g$ is another Riemannian metric on $M$ satisfying
$$
  \Vol_g(M)<\delta_n,
$$
then
\begin{equation}\label{eq:sabourau-intro}
  V_r(\widetilde M,\widetilde g)\ge V_r(\Hh^n)
  \qquad\text{for every }r\ge1.
\end{equation}    
\end{theorem}

Thus Sabourau obtained the desired conclusion simultaneously at every scale
$r\ge1$, but under a hypothesis involving the absolute volume of $(M,g)$.

Guth proved a complementary result in which the hypothesis is normalized by
simplicial volume, but the conclusion is obtained at one fixed scale
\cite[Theorem~2]{Guth}.  

\begin{theorem}[{\cite{Guth}}]
For every $n \ge 2$, there exists $\delta_n > 0$ such that
if $M$ is a closed, oriented, connected $n$-dimensional manifold admitting a hyperbolic metric,
and $g$ is another Riemannian metric on $M$ with
\begin{equation}\label{eq:guth-fixed-scale-hypothesis}
  \frac{\Vol_g(M)}{\|M\|_\Delta}<\delta_n,
\end{equation}
then 
$$
V_1(\widetilde{M}) \ge V_1(\mathbb{H}^n).
$$
\end{theorem}

The ratio in \eqref{eq:guth-fixed-scale-hypothesis} is invariant under finite
coverings, which makes it the natural hypothesis for a covering-space argument.

The strongest conjectural statement in this direction is Guth's
Conjecture \cite[Conjecture~2]{Guth}.  Let $g_{\mathrm{hyp}}$ be a
hyperbolic metric on $M$.  Guth conjectured that the strict volume inequality
\begin{equation}\label{eq:guth-conjecture-two-volume}
  \Vol_g(M)<\Vol_{g_{\mathrm{hyp}}}(M)
\end{equation}
implies the strict ball-volume comparison
\begin{equation}\label{eq:guth-conjecture-two-comparison}
  V_r(\widetilde M,\widetilde g)>V_r(\Hh^n)
  \qquad\text{for every }r>0.
\end{equation}
This conjecture strengthens the theorem of Besson--Courtois--Gallot, which,
under the same volume hypothesis, gives the comparison in
\eqref{eq:guth-conjecture-two-comparison} for all sufficiently large radii,
with the threshold allowed to depend on the metric $g$ \cite{BCG}.

A natural intermediate problem, also suggested by Guth \cite[P. 54]{Guth}, is to retain the
covering-invariant linear hypothesis from his fixed-scale theorem while
replacing its radius-one conclusion by the scale-uniform conclusion of
Sabourau's theorem.  In other words, one asks whether the condition
$$
  \frac{\Vol_g(M)}{\simp{M}}<\delta_n
$$
is sufficient to guarantee the hyperbolic ball-volume comparison for every
$r\ge1$.  This is the content of the following conjecture.

\begin{conjecture}[\cite{Guth}]\label{conj}
 For every $n\ge2$, there exists $\delta_n>0$ such that, if $M$ is a closed
oriented connected $n$-dimensional manifold admitting a hyperbolic metric and
$g$ is another Riemannian metric on $M$ with
\begin{equation}\label{eq:conjecture-four-hypothesis}
  \frac{\Vol_g(M)}{\simp{M}}<\delta_n,
\end{equation}
then
\begin{equation}\label{eq:conjecture-four-conclusion}
  V_r(\widetilde M,\widetilde g)\ge V_r(\Hh^n)
  \qquad\text{for every }r\ge1.
\end{equation}   
\end{conjecture}

The case $n=2$ was proved by Karam \cite{Karam}.  In arbitrary dimension,
Alpert recently established the all-scale conclusion under the stronger quadratic
hypothesis \cite[Theorem~1]{AlpertGrowth}: for every $n\ge2$, there exists
$\delta_n>0$ such that
\begin{equation}\label{eq:alpert-hypothesis}
  \frac{\Vol_g(M)^2}{\simp{M}}<\delta_n
\end{equation}
implies
\begin{equation}\label{eq:alpert-conclusion}
  V_r(\widetilde M,\widetilde g)\ge V_r(\Hh^n)
  \qquad\text{for every }r\ge1.
\end{equation}
Alpert's result therefore recovers the desired conclusion at every radius, but
the quadratic dependence on $\Vol_g(M)$ does not give
\eqref{eq:conjecture-four-hypothesis}.  To bridge this gap, Alpert proposed a
finite-scale logarithmic estimate for simplicial volume
\cite[Question~13]{AlpertGrowth}.  The main result of the present paper proves
a regularized form of that estimate.

\subsection{Main Result}

\begin{theorem}[Finite-scale logarithmic simplicial-volume estimate]
\label{thm:finite-scale-intro}
For every integer $n\ge2$, there exists a constant $C_n>0$ such that every
closed oriented connected Riemannian $n$-manifold $N$ satisfies
\begin{equation}\label{eq:finite-scale-intro}
  \simp{N}
  \le
  C_n\frac{\Vol(N)}{R^n}
  \brak{\log\paren{e+\frac{V_R(N)}{R^n}}}^{n}
\end{equation}
for every
$$
  0<R<\frac12\sys(N).
$$
\end{theorem}

As an application, the finite-scale estimate implies Conjecture \ref{conj} via standard residual finiteness and large systole finite covering arguments. For completeness, the deduction uses the following standard mechanism.  For a
fixed radius $r\ge1$, residual finiteness of the fundamental group of a closed
hyperbolic manifold provides a finite cover $N\to M$ with
$\sys(N)>2r$.  Radius-$r$ balls in $N$ then lift isometrically to the universal
cover.  Applying \cref{thm:finite-scale-intro} to $N$ at scale $r$, using the
multiplicativity of Riemannian volume and simplicial volume under finite
coverings, and comparing with the exponential upper bound for
$V_r(\Hh^n)$ gives a dimensional lower bound for
$\Vol_g(M)/\simp{M}$ whenever the desired ball comparison fails.  Choosing
$\delta_n$ below this lower bound proves
\cref{thm:universal-cover-intro}. We record the
consequence as follows.

\begin{theorem}
\label{thm:universal-cover-intro}
Conjecture \ref{conj} holds true.
\end{theorem}

\subsection{Proof Structure}

The proof of \cref{thm:finite-scale-intro} has four main steps.  First, we choose
an $R$-separating filtration
$$
  N=Z_n\supset Z_{n-1}\supset\cdots\supset Z_1\supset Z_0.
$$
When $R<\tfrac12\sys(N)$, the filtration-to-coloring construction together with
Gromov's Amenable Reduction Lemma gives
\begin{equation}\label{eq:intro-zero-stratum}
  \simp{N}\le 2^n\#Z_0.
\end{equation}
It therefore suffices to bound the cardinality of the finite zero-stratum
$Z_0$.

Second, we prove a cluster version of the near-minimizing separator estimate.
For a finite subset $A\subset Z_0$, provided the relevant neighborhood of $A$
is contained in an $R$-ball, one has
\begin{equation}\label{eq:intro-cluster}
  \#A\,\frac{t^n}{n!}
  \le \Vol\bigl(\cN_t^N(A)\bigr)+\eta.
\end{equation}
The same filtration is chosen so that this estimate holds simultaneously for
all admissible clusters $A$.  The key point is that the separator can be sliced
by a facewise affine approximation of the distance to $A$ function while
preserving the admissibility and purity of the replacement separator.

Third, a preliminary application of \eqref{eq:intro-cluster} gives a local
cardinality bound for $Z_0$.  We then apply a local padded-partition lemma to the
finite metric space $Z_0$.  It produces clusters of controlled diameter and
padded cores containing at least half of the points, with padding radius
$$
  \tau\asymp_n
  \frac{R}{\log\paren{e+V_R(N)/R^n}}.
$$
The ambient $\tau$-neighborhoods of distinct padded cores are disjoint.

Finally, we apply \eqref{eq:intro-cluster} to every padded core and sum the
resulting inequalities.  Disjointness gives
$$
  \#Z_0\,\tau^n\lesssim_n\Vol(N).
$$
Combining this estimate with \eqref{eq:intro-zero-stratum} yields
\eqref{eq:finite-scale-intro}.   The proof of
\cref{thm:universal-cover-intro} is then completed by the finite-cover argument
described above.

\subsection{Organization of the Paper}

In \cref{sec:preliminaries} we review separating filtrations, simplicial volume,
amenable reduction, and the analytic and polyhedral tools used later.  In
\cref{sec:cluster-slicing} we establish the cluster-slicing estimate for a
near-minimizing separator, and in \cref{sec:zero-cluster} we iterate that
estimate to obtain a simultaneous inequality for every subset of a single
zero-stratum.  The local padded-partition lemma is proved in
\cref{sec:padded}.  These ingredients are combined in
\cref{sec:finite-scale} to prove \cref{thm:finite-scale-intro}.  In
\cref{sec:large-systole} we construct finite covers with large metric systole
and record the ball-lifting lemma.  The universal-cover comparison, and hence
Conjecture \ref{conj}, is proved in \cref{sec:universal-cover}.  

\section{Separating filtrations and simplicial volume}\label{sec:preliminaries}

\subsection{Notation}
Unless explicitly stated otherwise, all base manifolds and all
polyhedra considered below are compact; universal covers need not be compact.  For a Riemannian manifold $(X,h)$, let $d_h$ denote the induced distance and set
$$
  B_h(x,r)=\set{y\in X:d_h(x,y)<r},
  \qquad
  \dist_h(x,A)=\inf_{a\in A}d_h(x,a).
$$
For a subset $A\subset X$, we use the extended-value convention
$$
  \dist_h(x,\varnothing)=+\infty
$$
and define its open $t$-neighborhood by
$$
  \cN_t^h(A)=\set{x\in X:\dist_h(x,A)<t}.
$$
Thus $\cN_t^h(\varnothing)=\varnothing$ for every finite $t$, and in particular $\cN_0^h(A)=\varnothing$.  When the metric is understood from the ambient space, we write $d_X$, $B_X$, $\dist_X$, and $\cN_t^X$ in place of $d_h$, $B_h$, $\dist_h$, and $\cN_t^h$, respectively, and we write $\Vol(E)$ in place of $\Vol_h(E)$ for a measurable subset $E\subset X$.  All balls and neighborhoods are open.  Hausdorff measures on a Riemannian polyhedron are denoted by $\Area_k$; on an $n$-manifold, $\Area_n$ agrees with the Riemannian volume.  For paths, $*$ denotes concatenation, $\overline{\beta}$ denotes the reverse of a path $\beta$, and $\operatorname{length}(\beta)$ denotes the length of $\beta$ in the understood metric.

\subsection{Riemannian polyhedra and separating subpolyhedra}
We use the polyhedral category introduced in \cite{AlpertSeparating}, which in turn builds on the separating-set method of Papasoglu \cite{Papasoglu}; see also \cite{Nabutovsky}.

A finite polyhedron $P$ is called a \emph{pure $d$-dimensional Riemannian polyhedron} when it is pure of dimension $d$ and its maximal simplices carry smooth Riemannian metrics agreeing on common faces.  Whenever such a polyhedron is smoothly embedded in a Riemannian manifold, we equip it facewise with the metric induced from the ambient manifold.  Its intrinsic polyhedral length distance then dominates the ambient distance; in particular, the restriction to $P$ of an ambient distance function is $1$-Lipschitz.

Throughout, a piecewise smooth path in a finite Riemannian polyhedron means a continuous path admitting a finite subdivision of its parameter interval such that every restricted path is smooth and is contained in one closed simplex.  The intrinsic polyhedral distance is the length distance obtained by taking the infimum over paths of this kind.

A \emph{codimension-one subpolyhedron} $Z\subset P$ is a pure $(d-1)$-dimensional polyhedron that is a subcomplex of a smooth triangulation refining $P$.  It is also required to satisfy the \emph{smooth-nesting condition}: the relative interior of every face of $Z$ lies in the relative interior of a face of the original polyhedron $P$ of strictly larger dimension.

For $R>0$, a codimension-one subpolyhedron $Z\subset P$ is \emph{$R$-separating} if every connected component of $P\setminus Z$ is contained in an ambient metric ball of radius $R$.  An \emph{$R$-separating filtration} of a closed $n$-manifold $N$ is a chain
$$
  N=Z_n\supset Z_{n-1}\supset\cdots\supset Z_1\supset Z_0
$$
in which $Z_i$ is an $R$-separating codimension-one subpolyhedron of $Z_{i+1}$.

\begin{convention}\label{conv:empty}
If every connected component of $P$ is already contained in an $R$-ball, we allow the empty set as an $R$-separating subpolyhedron of $P$ and assign it area zero.  This convention is equivalent to adjoining the empty competitor to the infimum and only removes a vacuous edge case.  It is compatible with the filtration and coloring arguments.
\end{convention}

For a pure $d$-polyhedron $P$, define
$$
  I_R(P)=\inf\set{\Area_{d-1}(Y):Y\subset P\text{ is $R$-separating}}.
$$
We say that $Z$ is \emph{$\eps$-minimizing} if
$$
  \Area_{d-1}(Z)\le I_R(P)+\eps.
$$
The infimum need not be attained; approximate minimizers are sufficient throughout.

\subsection{Simplicial volume and amenable reduction}
For a real singular chain $c=\sum_j a_j\sigma_j$, its $\ell^1$-norm is
$$
  \norm{c}_1=\sum_j\abs{a_j}.
$$
If $\alpha\in H_k(X;\R)$, the simplicial seminorm of $\alpha$ is
$$
  \norm{\alpha}_1
  =\inf\set{\norm{c}_1:c\text{ is a real singular cycle representing }\alpha}.
$$
For a closed oriented connected $n$-manifold $N$, its \emph{simplicial volume} is
$$
  \simp{N}=\norm{[N]_{\R}}_1.
$$
We refer to \cite{Gromov82,LoehSurvey} for the basic theory, including homotopy invariance and multiplicativity under finite coverings.

The filtration argument uses Gromov's Amenable Reduction Lemma.  Briefly, a singular cycle can be regarded as the image of a finite $\Delta$-complex obtained by gluing its singular simplices along matching faces.  A vertex coloring is called \emph{amenable} when, for every color, the subgroup generated by all monochromatic edges is amenable; the standard formulation also excludes an edge whose two endpoints are the same vertex.  A top-dimensional simplex is \emph{rainbow} if all its vertices have distinct colors.  Amenable reduction bounds the simplicial norm of the represented class by the total absolute coefficient of the rainbow simplices \cite[Section~3.2]{Gromov09}; see also \cite{AlpertKatz}.

\subsection{The zero-stratum controls simplicial volume}
The following elementary extension records explicitly what happens when one of the filtration levels is empty.

\begin{lemma}[Empty-level filtration coloring]\label{lem:empty-level-coloring}
Let
$$
  N=Z_n\supset Z_{n-1}\supset\cdots\supset Z_0
$$
be an $R$-separating filtration of a closed $n$-manifold, with the convention of \cref{conv:empty}.  If $Z_j=\varnothing$ for some $j$, then there is a finite triangulation of $N$ and a coloring of the vertices of its barycentric subdivision with the following properties:
\begin{enumerate}[label=(\roman*),leftmargin=2.5em]
\item no top-dimensional simplex is rainbow;
\item for every color, the union of the edges whose two endpoints have that color is contained in one connected component of
$$
  Z_i\setminus Z_{i-1}
$$
for some $i\in\{0,\ldots,n\}$, where $Z_{-1}=\varnothing$.
\end{enumerate}
\end{lemma}

\begin{proof}
Choose a common finite triangulation $K$ adapted to all nonempty filtration levels, as in the refinement step in the proof of \cite[Lemma~4]{AlpertSeparating}; an empty level imposes no additional condition.  Thus every nonempty $Z_i$ is a subcomplex of $K$.  For each simplex $\sigma$ of $K$, let
$$
  i(\sigma)=\min\set{i:\sigma\subset Z_i}.
$$
This index exists because $Z_n=N$.  Since $Z_{i(\sigma)-1}$ is a subcomplex and $\sigma$ is not contained in it, the relative interior $\sigma^\circ$ is contained in
$$
  Z_{i(\sigma)}\setminus Z_{i(\sigma)-1}.
$$
Let $C(\sigma)$ be the connected component of this difference that contains the barycenter $b_\sigma$.  Color the vertex $b_\sigma$ of the barycentric subdivision $\operatorname{sd}K$ by the pair
$$
  \bigl(i(\sigma),C(\sigma)\bigr).
$$

We first verify the compatibility that is implicit in the usual filtration coloring.  Suppose that $\sigma\subset\tau$ and
$$
  i(\sigma)=i(\tau)=i.
$$
For $0<s\le1$, the point
$$
  (1-s)b_\sigma+s b_\tau
$$
lies in $\tau^\circ$.  Because $Z_{i-1}$ is a subcomplex and $\tau\not\subset Z_{i-1}$, this relative interior is disjoint from $Z_{i-1}$.  The endpoint $b_\sigma$ is also outside $Z_{i-1}$ by the definition of $i(\sigma)$.  Hence the whole barycentric edge $[b_\sigma,b_\tau]$ is contained in
$$
  Z_i\setminus Z_{i-1}.
$$
It follows that $C(\sigma)=C(\tau)$.  Consequently, two vertices of one barycentric simplex that have the same level index have the same color, and every monochromatic edge is contained in the corresponding component $C(\sigma)$.  This proves (ii).

A top-dimensional simplex of $\operatorname{sd}K$ has $n+1$ vertices and corresponds to a strictly nested chain of simplices of $K$.  If it were rainbow, the compatibility just proved would force its $n+1$ vertices to have pairwise distinct level indices.  But $Z_j=\varnothing$ implies that no simplex has level index $j$; indeed, the indices $0,\ldots,j$ are all absent.  Thus fewer than $n+1$ level indices occur, and a rainbow $n$-simplex is impossible.  This proves (i).
\end{proof}

The following proposition is the precise result from the amenable-reduction and filtration machinery.

\begin{proposition}[Alpert--Gromov, trivial-image form]\label{prop:zero-stratum}
Let $N$ be a closed oriented connected $n$-manifold with an $R$-separating filtration
$$
  N=Z_n\supset Z_{n-1}\supset\cdots\supset Z_0.
$$
Suppose that, for every metric $R$-ball $B\subset N$, the homomorphism
$$
  \pi_1(B)\longrightarrow\pi_1(N)
$$
induced by inclusion is trivial.  Then
\begin{equation}\label{eq:zero-stratum}
  \simp{N}\le 2^n\#Z_0.
\end{equation}
\end{proposition}

\begin{proof}
First suppose that every level of the filtration is nonempty.  The construction in \cite[Lemma~4]{AlpertSeparating} produces a triangulation adapted to the filtration and a vertex coloring for which precisely $2^n\#Z_0$ top-dimensional simplices are rainbow.  It also shows that, for each color, the monochromatic edge complex lies in one connected component of a difference $Z_i\setminus Z_{i-1}$.

If some level is empty, then $Z_0=\varnothing$, and \cref{lem:empty-level-coloring} gives the same containment property with no rainbow top-dimensional simplex.  Thus, in either case, the rainbow mass is at most $2^n\#Z_0$.

For $i\ge1$, every connected component of $Z_i\setminus Z_{i-1}$ is contained in an $R$-ball because $Z_{i-1}$ is $R$-separating in $Z_i$; for $i=0$, every component is a point.  Hence, by the hypothesis on $R$-balls, the homomorphism to $\pi_1(N)$ induced by every monochromatic edge complex is trivial.  In particular, the associated subgroup is amenable.  Gromov's Amenable Reduction Lemma bounds the simplicial norm of the fundamental class by the rainbow mass \cite[Section~3.2]{Gromov09}; an explicit formulation and proof also appear in \cite{AlpertKatz}.  The oriented triangulation represents the fundamental class with coefficients of absolute value one, and \eqref{eq:zero-stratum} follows.
\end{proof}

\begin{remark}
An amenable-image extension is stated in \cite[Lemma~8]{AlpertGrowth}.  The present manuscript only uses the trivial image case above, so no amenable extension of Alpert's filtration-to-coloring argument is needed here.
\end{remark}

We will use the following elementary consequence of the systole assumption.

\begin{lemma}\label{lem:ball-pi1}
If $0<R<\tfrac12\sys(N)$, then the inclusion of every $R$-ball in $N$ induces the trivial homomorphism on fundamental groups.
\end{lemma}

\begin{proof}
Every element of the fundamental group of the open set $B_N(p,R)$ has a piecewise smooth representative, so let $\gamma$ be such a loop.  Compactness of its image gives $R'<R$ with
$$
  \gamma(S^1)\subset B_N(p,R').
$$
Parametrize $\gamma$ by arc length and subdivide it cyclically into consecutive arcs $\alpha_0,\ldots,\alpha_{m-1}$ of length smaller than $2(R-R')$.  If $x_i$ is the initial point of $\alpha_i$, set $x_m=x_0$, choose a minimizing path $\beta_i$ from $p$ to $x_i$, and set $\beta_m=\beta_0$.  Each based loop
$$
  \beta_i*\alpha_i*\overline{\beta_{i+1}}
  \qquad(0\le i<m)
$$
has length strictly smaller than
$$
  R'+2(R-R')+R'=2R<\sys(N),
$$
and hence is null-homotopic in $N$.  Their product is
$$
  \beta_0*\gamma*\overline{\beta_0},
$$
so $\gamma$ is null-homotopic in $N$.  Thus the inclusion-induced homomorphism is trivial.
\end{proof}

\subsection{Analytic and polyhedral results}
We use two standard analytic facts and prove the two elementary polyhedral facts needed in the replacement argument.

\begin{enumerate}[label=(\roman*),leftmargin=2.5em]
\item If $f$ is Lipschitz on a finite-dimensional Riemannian manifold, then for every pair $\alpha,\beta>0$ it has a smooth approximation $h$ satisfying
$$
  \abs{h-f}<\alpha,
  \qquad
  \Lip(h)\le \Lip(f)+\beta.
$$
Indeed, one applies the approximation theorem of \cite{AzagraEtAl} with uniform tolerance $\alpha/2$ and Lipschitz loss $\beta$; its non-strict uniform estimate then implies the displayed strict inequality.

\item The coarea formula is applied face by face on a Riemannian polyhedron.  In particular, if $u$ is $(1+\delta)$-Lipschitz on a pure $d$-polyhedron $P$, then
\begin{equation}\label{eq:coarea-prelim}
  \int_J \Area_{d-1}\paren{u^{-1}(t)}\,dt
  \le (1+\delta)\Area_d\paren{u^{-1}(J)}
\end{equation}
for every interval $J$.  See, for example, \cite[Theorem~13.4.2]{BuragoZalgaller}.
\end{enumerate}

\begin{lemma}[Facewise differential bound]\label{lem:facewise-lipschitz}
Let $K$ be a finite smooth triangulation of a Riemannian polyhedron $P$, and let $u:P\to\R$ be continuous and $C^1$ on every closed simplex.  Suppose that $L\ge0$ and that
$$
  \sup_{x\in\sigma}\norm{d(u|_\sigma)_x}\le L
$$
for every positive-dimensional simplex $\sigma$ of $K$.  Then $u$ is $L$-Lipschitz with respect to the intrinsic polyhedral distance.
\end{lemma}

\begin{proof}
Let $\gamma:[a,b]\to P$ be a piecewise smooth path.  By the convention above, there is a finite subdivision
$$
  a=t_0<t_1<\cdots<t_m=b
$$
and closed simplices $\sigma_1,\ldots,\sigma_m$ such that $\gamma([t_{r-1},t_r])\subset\sigma_r$ and each restricted path is smooth.  The ordinary chain rule on the closed simplex $\sigma_r$ gives
$$
\begin{aligned}
  \abs{u(\gamma(t_r))-u(\gamma(t_{r-1}))}
  &\le
  \int_{t_{r-1}}^{t_r}
  \norm{d(u|_{\sigma_r})_{\gamma(t)}}\,\norm{\gamma'(t)}\,dt \\
  &\le L\,\operatorname{length}\paren{\gamma|_{[t_{r-1},t_r]}}.
\end{aligned}
$$
Summing over $r$ yields
$$
  \abs{u(\gamma(b))-u(\gamma(a))}
  \le L\,\operatorname{length}(\gamma).
$$
Taking the infimum over such paths proves the assertion.
\end{proof}

The next lemma permits us to replace a smooth function by a piecewise-affine function without losing either uniform or Lipschitz control.  Its proof is included because this is the point at which compatibility with the fixed polyhedral structure is needed.

\begin{lemma}[Facewise-affine approximation]\label{lem:facewise-affine}
Let $K$ be a finite smooth triangulation of a compact Riemannian polyhedron $P$, and let $h:P\to\R$ be continuous and smooth on every closed simplex of $K$.  For every $\alpha,\beta>0$ there are a finite smooth refinement $K'$ of $K$ and a continuous function $\ell:P\to\R$, affine in barycentric coordinates on every simplex of $K'$, such that
\begin{equation}\label{eq:affine-approximation}
  \abs{\ell-h}<\alpha,
  \qquad
  \Lip(\ell)\le
  \max_{\substack{\sigma\in K\\ \dim\sigma>0}}
  \sup_{x\in\sigma}\norm{d(h|_\sigma)_x}+\beta.
\end{equation}
The differential norm is computed with the facewise Riemannian metric; if $P$ is zero-dimensional, the maximum is understood as zero.
\end{lemma}

\begin{proof}
For every closed $k$-simplex $\sigma$ of $K$, choose its smooth characteristic map
$$
  \phi_\sigma:\Delta^k\longrightarrow\sigma
$$
and put $g_\sigma=h\circ\phi_\sigma$.  The pullback of the facewise metric is a smooth positive-definite metric on the compact simplex $\Delta^k$.

For an integer $m\ge1$, take the standard $m$th edgewise subdivision of every abstract simplex.  These subdivisions agree on common faces, their mesh tends to zero as $m\to\infty$, and they are uniformly shape regular in each fixed dimension: if
$$
  S=[v_0,\ldots,v_k]
$$
is a $k$-simplex in the subdivision of $\Delta^k$ and $A_S$ is the matrix with columns $v_i-v_0$, then
\begin{equation}\label{eq:shape-regular}
  \norm{A_S^{-T}}\,\diam S\le c_k
\end{equation}
for a constant $c_k$ independent of $m$ and $S$.  This follows directly from the usual edgewise construction, because after rescaling there are only finitely many simplex shapes in each dimension.

At a subdivision vertex represented by $v\in\Delta^k$, define
$$
  \ell_m\bigl(\phi_\sigma(v)\bigr)=g_\sigma(v)
$$
and extend affinely in the barycentric coordinates of every subdivided simplex.  These definitions agree on common faces, so $\ell_m$ is continuous on $P$.  If $x\in S$ has barycentric coordinates $\lambda_0,\ldots,\lambda_k$, then
$$
  (\ell_m\circ\phi_\sigma)(x)-g_\sigma(x)
  =\sum_{i=0}^k\lambda_i\bigl(g_\sigma(v_i)-g_\sigma(x)\bigr).
$$
Uniform continuity of the finitely many functions $g_\sigma$ therefore gives
\begin{equation}\label{eq:uniform-interpolation}
  \norm{\ell_m-h}_{L^\infty(P)}\longrightarrow0.
\end{equation}

We also record the derivative estimate.  Let $p_S$ be the Euclidean gradient of the affine interpolant of $g_\sigma$ on $S$.  For $x\in S$ and $1\le i\le k$, the fundamental theorem of calculus gives
\begin{align*}
  \left\langle p_S-\nabla g_\sigma(x),v_i-v_0\right\rangle
  &={}g_\sigma(v_i)-g_\sigma(v_0)
      -\left\langle\nabla g_\sigma(x),v_i-v_0\right\rangle\\
  &=\int_0^1
    \left\langle
      \nabla g_\sigma\bigl(v_0+s(v_i-v_0)\bigr)-\nabla g_\sigma(x),
      v_i-v_0
    \right\rangle ds.
\end{align*}
Let $\omega_\sigma(r)$ be the modulus of continuity of $\nabla g_\sigma$.  Combining the last display with \eqref{eq:shape-regular} yields
\begin{equation}\label{eq:gradient-interpolation}
  \sup_{x\in S}\norm{p_S-\nabla g_\sigma(x)}
  \le c_k\sqrt{k}\,\omega_\sigma(\diam S).
\end{equation}
Hence the differentials of the affine interpolants converge uniformly, on every old simplex and in every face dimension, to the differential of $h$.  Since there are only finitely many simplices and the Euclidean and pullback Riemannian covector norms are uniformly equivalent on each of them, \eqref{eq:gradient-interpolation} implies that, for all sufficiently large $m$,
$$
  \sup_{x\in\tau}\norm{d(\ell_m|_\tau)_x}
  \le
  \max_{\substack{\sigma\in K\\ \dim\sigma>0}}
  \sup_{x\in\sigma}\norm{d(h|_\sigma)_x}+\beta
$$
for every positive-dimensional simplex $\tau$ of the edgewise refinement.

Applying \cref{lem:facewise-lipschitz} to the edgewise refinement gives the same upper bound for $\Lip(\ell_m)$ with respect to the intrinsic polyhedral distance.  Choose $m$ large enough that this estimate and \eqref{eq:uniform-interpolation} with error $<\alpha$ both hold, and set $K'=K^{(m)}$ and $\ell=\ell_m$.

Finally, every new closed simplex is affine in the barycentric coordinates of one old simplex, and its characteristic map into $P$ is the composition of that affine embedding with the old smooth characteristic map.  Thus $K'$ is a finite smooth refinement of $K$.
\end{proof}

We next record the elementary convex-geometric fact that controls the carrier faces created by affine cuts.

\begin{lemma}[Carrier faces for affine cuts]\label{lem:affine-carrier-faces}
Let $\tau$ be a $k$-simplex in an affine space, where $k\ge1$.
\begin{enumerate}[label=(\roman*),leftmargin=2.5em]
\item If $F$ is a face of $\tau$ and $H$ is an affine hyperplane, then $F\cap H$ is either empty or a face of $\tau\cap H$.
\item Let $\mathcal A$ be a finite family of affine hyperplanes, and let $\mathcal Q(\tau,\mathcal A)$ be the polytopal complex obtained by cutting $\tau$ by all hyperplanes in $\mathcal A$ and taking all faces.  Let $\sigma$ be a face of $\tau$, let $C$ be a cell of the induced subdivision of $\sigma$, and suppose that for some $H\in\mathcal A$,
$$
  C\subset H,
  \qquad
  H\cap\sigma^\circ\ne\varnothing,
  \qquad
  \sigma\not\subset H.
$$
Then $C$ is a face of a $(k-1)$-dimensional cell of $\mathcal Q(\tau,\mathcal A)$ contained in $\tau\cap H$.
\end{enumerate}
\end{lemma}

\begin{proof}
For (i), write the face $F$ as the locus where a supporting affine functional for $\tau$ attains its minimum.  Restricting that functional to $\tau\cap H$ shows that $F\cap H$ is a face of the section.

For (ii), first note that a finite affine-hyperplane subdivision of a convex polytope $Q$ is pure of dimension $\dim Q$.  Hyperplanes that contain the affine hull of $Q$ do not cut $Q$ and may be discarded.  Let $D$ be any cell and choose $x\in D^\circ$.  The union of the remaining hyperplane traces has empty relative interior in $Q$, so points of $Q^\circ$ outside all those traces occur arbitrarily close to $x$.  Choose a sequence of such points converging to $x$.  There are only finitely many full-dimensional chambers, so a subsequence lies in one chamber $E^\circ$.  Then $x\in E$.  Since intersections of cells in a polytopal complex are common faces, $D\cap E$ is a face of $D$ containing $x\in D^\circ$; hence $D\cap E=D$, and $D$ is a face of the full-dimensional cell $E$.

The hypotheses on $H$ imply that the restriction of an affine defining function for $H$ changes sign on $\sigma$ near a point of $H\cap\sigma^\circ$.  Thus $H$ cuts the relative interior of $\tau$ and is not a supporting hyperplane of $\tau$.  Consequently,
$$
  \dim(\tau\cap H)=k-1.
$$
The traces on $\tau\cap H$ of the hyperplanes in $\mathcal A$ form a finite affine-hyperplane subdivision of this $(k-1)$-polytope.  Because the subdivision restricts compatibly to every face of $\tau$, the cell $C$ is a cell of this restricted complex.  Applying the preceding purity argument inside $\tau\cap H$ shows that $C$ is a face of a $(k-1)$-dimensional cell contained in $\tau\cap H$.
\end{proof}

We next subdivide a piecewise-affine level.  Unlike a smooth regular-level argument, this construction is completely relative to the given triangulation and is smooth on every closed new simplex.

\begin{lemma}[Subdivision by a piecewise-affine level]\label{lem:pl-level-subdivision}
Let $K$ be a finite smooth triangulation of a pure $d$-dimensional polyhedron $P$, where $d\ge1$, let $L\subset K$ be a subcomplex, and let $\ell:P\to\R$ be continuous and affine on every simplex of $K$.  If
$$
  t\notin\set{\ell(v):v\text{ is a vertex of }K},
$$
then $K$ has a finite smooth refinement $K_t$ such that the induced refinement of $L$ is a subcomplex of $K_t$ and each of
$$
  \ell^{-1}(t),\qquad
  \set{\ell\le t},\qquad \set{\ell\ge t}
$$
is the support of a subcomplex of $K_t$.  Moreover, $\ell^{-1}(t)$ is either empty or pure of dimension $d-1$.  More precisely, if a simplex of the level subcomplex has relative interior contained in the relative interior of a $k$-simplex of $K$, then its dimension is at most $k-1$.
\end{lemma}

\begin{proof}
In the affine reference coordinates of every closed simplex $\sigma$ of $K$, cut $\sigma$ by the hyperplane $\ell=t$ and take all nonempty faces of the three convex polytopes
$$
  \sigma\cap\set{\ell\le t},\qquad
  \sigma\cap\set{\ell=t},\qquad
  \sigma\cap\set{\ell\ge t}.
$$
Because the restrictions of $\ell$ agree on common faces, these polytopes and all of their faces form a finite polytopal complex $\mathcal Q_t$ refining $K$.  Take the barycentric subdivision of $\mathcal Q_t$.  Each resulting simplex is an affine simplex in the reference coordinates of a single old simplex of $K$.  Composing its affine inclusion with the old characteristic map therefore makes it a smooth closed simplex in $P$.  The resulting complex $K_t$ is a finite smooth refinement of $K$.

Every old simplex, every simplex of $L$, the level, and each closed side are unions of cells of $\mathcal Q_t$; hence the induced refinement of $L$ is a subcomplex of $K_t$, while the level and each closed side are supports of subcomplexes of $K_t$.  The dimension assertion follows immediately: inside the relative interior of a $k$-simplex on which the level is nonempty, the affine hyperplane section has dimension $k-1$.

It remains to check purity.  Let $Q$ be any cell of the level polytopal complex, and let $\sigma$ be its unique old carrier, characterized by
$$
  Q^\circ\subset\sigma^\circ.
$$
Put $H=\set{\ell=t}$ in the affine reference simplex containing $\sigma$.  Then $Q\subset H$ and $H\cap\sigma^\circ\ne\varnothing$.  Moreover, $\sigma\not\subset H$, because otherwise all vertices of $\sigma$ would have value $t$, contrary to the hypothesis on $t$.  Choose a $d$-simplex $\tau$ of $K$ containing $\sigma$, which is possible because $P$ is pure.  Applying \cref{lem:affine-carrier-faces} to the one-hyperplane arrangement $\{H\}$ shows that $Q$ is a face of a $(d-1)$-dimensional cell contained in $\tau\cap H$.  Thus every level cell is a face of a $(d-1)$-dimensional level cell.  The level polytopal complex is therefore pure of dimension $d-1$, and barycentric subdivision preserves purity.
\end{proof}

The preceding subdivision also gives a direct nonemptiness argument for the class of separating competitors.

\begin{lemma}[Existence of separating competitors]\label{lem:separators-exist}
Let $P\subset M$ be a compact pure $d$-dimensional smooth Riemannian subpolyhedron of a Riemannian manifold, where $d\ge1$.  For every $R>0$, there exists an $R$-separating codimension-one subpolyhedron of $P$.  In particular,
$$
  0\le I_R(P)<\infty.
$$
\end{lemma}

\begin{proof}
Fix a finite smooth triangulation $K$ refining the original polyhedron structure of $P$, and let $V$ be its vertex set.  For $v\in V$, let
$$
  \lambda_v:P\longrightarrow[0,1]
$$
be the global barycentric coordinate of $v$: on a simplex not containing $v$ it is zero, and on a simplex containing $v$ it is the usual affine coordinate.  The map
$$
  \Lambda:P\longrightarrow[0,1]^V,
  \qquad
  \Lambda(x)=\bigl(\lambda_v(x)\bigr)_{v\in V},
$$
is a continuous injection.  Since $P$ is compact, it is a homeomorphism onto its image, and its inverse is uniformly continuous.  Choose an integer $m\ge2$ so large that
\begin{equation}\label{eq:barycentric-small-box}
  \norm{\Lambda(x)-\Lambda(y)}_\infty<\frac1m
  \quad\Longrightarrow\quad
  d_M(x,y)<R.
\end{equation}

In the affine reference simplex of every simplex of $K$, cut simultaneously by all hyperplanes
$$
  \lambda_v=\frac jm,
  \qquad
  v\in V,
  \quad 1\le j\le m-1,
$$
and take all faces of all resulting convex polytopes.  The cuts agree on common faces, so they form a finite polytopal complex refining $K$.  Its barycentric subdivision is a finite smooth refinement $K_m$: as in \cref{lem:pl-level-subdivision}, every new simplex is affine in one old reference simplex and therefore maps smoothly into $P$.

Let $Y$ be the union of all the grid levels
\begin{equation}\label{eq:grid-separator}
  Y=\bigcup_{v\in V}\bigcup_{j=1}^{m-1}
  \lambda_v^{-1}\paren{\frac jm}.
\end{equation}
Then $Y$ is a subcomplex of $K_m$.  It is pure of dimension $d-1$.  To see this, let $Q$ be any cell of the polytopal subcomplex supported on $Y$, and let $\sigma$ be its unique old carrier, so that $Q^\circ\subset\sigma^\circ$.  Choose one grid hyperplane
$$
  H=\set{\lambda_v=j/m}
$$
that contains $Q$.  Because $j/m\in(0,1)$ and the restriction of $\lambda_v$ to an old simplex is either zero or an ordinary barycentric coordinate, the restriction to $\sigma$ is not constant at the value $j/m$.  Hence
$$
  H\cap\sigma^\circ\ne\varnothing,
  \qquad
  \sigma\not\subset H.
$$
Choose a $d$-simplex $\tau$ of $K$ containing $\sigma$.  Apply \cref{lem:affine-carrier-faces} to the full finite grid arrangement in $\tau$.  It follows that $Q$ is a face of a $(d-1)$-dimensional arrangement cell contained in $\tau\cap H$, and this cell is contained in $Y$.  Thus every cell of $Y$ is a face of a $(d-1)$-dimensional cell of $Y$.  Barycentric subdivision preserves this property.

The smooth-nesting condition also follows from the carrier dimensions.  Let $\eta$ be a simplex of $Y$, and let $Q$ be the unique face of the original polyhedron structure of $P$ with $\eta^\circ\subset Q^\circ$.  Choose one grid level in \eqref{eq:grid-separator} containing $\eta$, and let $\sigma$ be the carrier of $\eta^\circ$ in $K$.  Because $K$ refines the original polyhedron structure, $\sigma^\circ\subset Q^\circ$.  The chosen barycentric coordinate is not constant at its grid value on $\sigma$, so
$$
  \dim\eta\le\dim\sigma-1<\dim Q.
$$
Thus $Y$ is a codimension-one subpolyhedron of $P$.

Finally, let $C$ be a connected component of $P\setminus Y$.  For every $v\in V$, continuity implies that $\lambda_v(C)$ is contained in one component of
$$
  [0,1]\setminus\set{\frac1m,\ldots,\frac{m-1}{m}}.
$$
Consequently,
$$
  \norm{\Lambda(x)-\Lambda(y)}_\infty<\frac1m
  \qquad(x,y\in C).
$$
By \eqref{eq:barycentric-small-box}, the ambient diameter of $C$ is strictly smaller than $R$.  Choosing any $p\in C$ gives $C\subset B_M(p,R)$.  Hence $Y$ is $R$-separating.  Its area is finite because it is a finite subpolyhedron, so $I_R(P)<\infty$.
\end{proof}

\section{Slicing near-minimizing separators by cluster neighborhoods}\label{sec:cluster-slicing}

This section contains the main geometric extension of the separator estimate in \cite[Lemma~6]{AlpertSeparating}.  The distance to one point is replaced by the distance to a finite set.

\begin{lemma}[Cluster slicing]\label{lem:cluster-slicing}
Let $P\subset M$ be a pure $d$-dimensional smooth Riemannian subpolyhedron of a closed Riemannian manifold $M$, equipped facewise with the metric induced from $M$, where $d\ge1$.  Let $Z\subset P$ be an $R$-separating codimension-one subpolyhedron satisfying
$$
  \Area_{d-1}(Z)\le I_R(P)+\eps.
$$
Let $A\subset P$ be finite, and let $0\le a<b<R$.  Suppose that there is $q\in M$ such that
\begin{equation}\label{eq:cluster-contained}
  \overline{\cN_b^M(A)}\subset B_M(q,R).
\end{equation}
Then
\begin{equation}\label{eq:cluster-slicing}
  \int_a^b \Area_{d-1}\paren{Z\cap\cN_t^M(A)}\,dt
  \le
  \Area_d\paren{P\cap\bigl(\cN_b^M(A)\setminus\cN_a^M(A)\bigr)}
  +2\eps R.
\end{equation}
\end{lemma}

\begin{proof}
If $A=\varnothing$, both neighborhood terms vanish and the conclusion is immediate.  Assume henceforth that $A$ is nonempty, and set
$$
  f(x)=\dist_M(x,A),\qquad x\in M.
$$
The function $f$ is $1$-Lipschitz on $M$.  Its restriction to $P$ is also $1$-Lipschitz for the intrinsic polyhedral metric because that metric dominates the ambient distance.  Fix a smooth triangulation $K$ of $P$ for which $Z$ is a subcomplex.

Choose
$$
  0<\delta<\min\set{\frac{b-a}{4},1}.
$$
First apply the smooth-approximation statement with uniform tolerance $\delta/4$ and Lipschitz loss $\delta/4$.  We obtain a smooth function $h=h_\delta:M\to\R$ satisfying
$$
  \abs{h-f}<\frac{\delta}{4},
  \qquad
  \Lip(h)\le1+\frac{\delta}{4}.
$$
Apply \cref{lem:facewise-affine} to $h|_P$, the triangulation $K$, uniform tolerance $\delta/4$, and differential loss $\delta/4$.  It gives a finite smooth refinement $K_\delta$ of $K$ and a continuous function $\ell=\ell_\delta:P\to\R$, affine on every simplex of $K_\delta$, such that
\begin{equation}\label{eq:affine-distance-approx}
  \abs{\ell-f}<\frac{\delta}{2}<\delta,
  \qquad
  \Lip(\ell)\le1+\frac{\delta}{2}<1+\delta.
\end{equation}
Because $K_\delta$ refines $K$, the separator $Z$ is still a subcomplex.  Write
$$
  \whU_t=\set{x\in P:\ell(x)<t},
  \qquad
  \whS_t=\set{x\in P:\ell(x)=t}.
$$
Then
\begin{equation}\label{eq:neighborhood-inclusions}
  P\cap\cN_{t-\delta}^M(A)
  \subset\whU_t
  \subset P\cap\cN_{t+\delta}^M(A).
\end{equation}

\smallskip
\noindent\emph{Good levels and the competitor.}
The set of values of $\ell$ at vertices of the finite complex $K_\delta$ is finite.  Fix
$$
  t\in[a+\delta,b-\delta]
$$
outside that finite set and define
\begin{equation}\label{eq:competitor}
  Z_t'=
  \paren{Z\cap\set{\ell\ge t}}
  \cup\whS_t.
\end{equation}
Apply \cref{lem:pl-level-subdivision} to $K_\delta$, with $L$ equal to the subcomplex representing $Z$.  The resulting finite smooth refinement has $Z$, the level, and both closed sides as supports of subcomplexes.  Hence $Z_t'$ is the support of a subcomplex of a smooth refinement of $P$.

\smallskip
\noindent\emph{Purity and smooth nesting.}
The level $\whS_t$ is empty or pure of dimension $d-1$ by \cref{lem:pl-level-subdivision}.  We verify that adjoining the retained part of $Z$ creates no lower-dimensional maximal simplex.  Let $\lambda$ be a simplex of $Z_t'$ of dimension less than $d-1$, and choose $x\in\lambda^\circ$.  If $\ell(x)=t$, then $\lambda$ is a level simplex and is a face of a $(d-1)$-simplex of the pure level subcomplex.  If $\ell(x)>t$, then $\lambda$ lies in the retained part of $Z$.  A subdivision of the pure complex $Z$ is pure, so points in relative interiors of $(d-1)$-simplices of $Z$ converge to $x$.  By continuity of $\ell$, those points may be chosen in the open set $\set{\ell>t}$.  Since the refined complex is finite, a subsequence lies in one fixed retained $(d-1)$-simplex whose closure contains $x$; the simplicial intersection property then makes $\lambda$ a face of that simplex.  Hence $Z_t'$ is either empty or pure of dimension $d-1$; in the empty case admissibility is understood according to \cref{conv:empty}.

To check smooth nesting, let $\lambda$ be a simplex of $Z_t'$, and let $Q$ be the unique face of the original polyhedron structure of $P$ whose relative interior contains $\lambda^\circ$.  If $\lambda$ lies in the level, let $\kappa$ be the unique simplex of $K_\delta$ whose relative interior contains $\lambda^\circ$.  Then $\kappa^\circ\subset Q^\circ$, and \cref{lem:pl-level-subdivision} gives
$$
  \dim\lambda\le\dim\kappa-1<\dim Q.
$$
Otherwise $\lambda^\circ\subset\set{\ell>t}$ and $\lambda$ lies in $Z$.  Let $\zeta$ be the unique face of the fixed original polyhedron structure of $Z$ whose relative interior contains $\lambda^\circ$.  The face $\zeta$ lies in the same original face $Q$ of $P$, and the smooth-nesting condition for $Z$ gives
$$
  \dim\lambda\le\dim\zeta<\dim Q.
$$
Thus $Z_t'$ satisfies smooth nesting relative to the original polyhedron structure of $P$.

\smallskip
\noindent\emph{The separating property.}
Let $C$ be a connected component of $P\setminus Z_t'$.  Since $\whS_t\subset Z_t'$, the continuous function $\ell-t$ has constant sign on $C$.

If $C\subset\whU_t$, then \eqref{eq:neighborhood-inclusions}, the chosen range of $t$, and \eqref{eq:cluster-contained} give
$$
  C\subset\cN_{t+\delta}^M(A)
  \subset\cN_b^M(A)
  \subset\overline{\cN_b^M(A)}
  \subset B_M(q,R).
$$
If $C\subset\set{\ell>t}$, then $C$ is a connected subset of $P\setminus Z$, because $Z_t'$ agrees with $Z$ on $\set{\ell>t}$.  Hence $C$ lies in a connected component of $P\setminus Z$, and therefore in an ambient $R$-ball.  If $Z_t'$ is nonempty, it is therefore an admissible $R$-separating codimension-one subpolyhedron.  If $Z_t'=\varnothing$, the same componentwise argument shows that every component of $P$ lies in an $R$-ball, so the empty competitor is admissible by \cref{conv:empty}.

Because $t$ is not a vertex value, $\ell$ is not identically equal to $t$ on any simplex of $K_\delta$.  Its level section in every positive-dimensional $(d-1)$-simplex of $Z$ therefore has dimension at most $d-2$; when $d=1$, the zero-dimensional complex $Z$ contains no point of the level.  In either case,
$$
  \Area_{d-1}(Z\cap\whS_t)=0.
$$
Consequently,
\begin{align*}
  \Area_{d-1}(Z)
  &=\Area_{d-1}(Z\cap\whU_t)
    +\Area_{d-1}(Z\cap\set{\ell>t}),\\
  \Area_{d-1}(Z_t')
  &\le
  \Area_{d-1}(Z\cap\set{\ell>t})
  +\Area_{d-1}(\whS_t).
\end{align*}
Since $Z_t'$ is an admissible competitor and
$$
  \Area_{d-1}(Z)\le I_R(P)+\eps,
$$
we obtain
\begin{equation}\label{eq:pointwise-slice}
  \Area_{d-1}(Z\cap\whU_t)
  \le
  \Area_{d-1}(\whS_t)+\eps
\end{equation}
for every $t\in[a+\delta,b-\delta]$ outside the finite set of vertex values.

Put
$$
  E_\delta=\set{x\in P:a+\delta<\ell(x)<b-\delta}.
$$
Integrating \eqref{eq:pointwise-slice} over the interval (the finite exceptional set has measure zero), using $b-a<R$, and applying coarea face by face to the Lipschitz function $\ell$ gives
\begin{align}
  \int_{a+\delta}^{b-\delta}
  \Area_{d-1}(Z\cap\whU_t)\,dt
  &\le
  \int_{a+\delta}^{b-\delta}
  \Area_{d-1}(\whS_t)\,dt+\eps R\notag\\
  &\le
  (1+\delta)\Area_d(E_\delta)
  +\eps R.\label{eq:integrated-affine}
\end{align}
The uniform approximation implies
$$
  E_\delta
  \subset
  P\cap\paren{\cN_b^M(A)\setminus\cN_a^M(A)}.
$$
It also gives
$$
  P\cap\cN_u^M(A)\subset\whU_{u+\delta}
  \qquad(a\le u\le b-2\delta).
$$
After the change of variables $t=u+\delta$,
$$
  \int_a^{b-2\delta}
  \Area_{d-1}(Z\cap\cN_u^M(A))\,du
  \le
  \int_{a+\delta}^{b-\delta}
  \Area_{d-1}(Z\cap\whU_t)\,dt.
$$
The remaining interval $[b-2\delta,b]$ has length $2\delta$, so its contribution is at most
$$
  2\delta\,\Area_{d-1}(Z).
$$
Combining this with \eqref{eq:integrated-affine} yields
\begin{align*}
  \int_a^b\Area_{d-1}(Z\cap\cN_u^M(A))\,du
  &\le
  (1+\delta)
  \Area_d\paren{P\cap\bigl(\cN_b^M(A)\setminus\cN_a^M(A)\bigr)}\\
  &\qquad +\eps R+2\delta\Area_{d-1}(Z).
\end{align*}
Choose any sequence $\delta_k\downarrow0$ satisfying the displayed upper bound and, for each $k$, choose the corresponding smooth and facewise-affine approximations.  The left-hand side and the shell on the right-hand side are independent of those choices.  Since $P$ and $Z$ have finite volume and area, passage to the limit gives the stronger estimate with $\eps R$ in place of $2\eps R$, and hence proves \eqref{eq:cluster-slicing}.  The argument includes $a=0$ under the convention $\cN_0^M(A)=\varnothing$.
\end{proof}

\begin{remark}\label{rem:why-cluster}
The topology of $A$ plays no role in \cref{lem:cluster-slicing}.  The only global condition used in the replacement step is that the entire region removed from the separator is contained in one $R$-ball.  This is why a disconnected finite cluster can be treated exactly as a single center.
\end{remark}

\section{A simultaneous zero-stratum estimate for all clusters}\label{sec:zero-cluster}

We now iterate \cref{lem:cluster-slicing} through a separating filtration.  

\begin{lemma}[Clustered zero-stratum estimate]\label{lem:cluster-zero}
Let $N$ be a closed Riemannian $n$-manifold, let $R>0$, and let $\eta>0$.  There is an $R$-separating filtration
$$
  N=Z_n\supset Z_{n-1}\supset\cdots\supset Z_1\supset Z_0
$$
with the following property.  For every subset $A\subset Z_0$ and every $0<t<R$ such that
\begin{equation}\label{eq:cluster-ball-condition}
  \overline{\cN_t^N(A)}\subset B_N(q,R)
\end{equation}
for some $q\in N$, one has
\begin{equation}\label{eq:cluster-zero}
  \#A\,\frac{t^n}{n!}
  \le \Vol\bigl(\cN_t^N(A)\bigr)+\eta.
\end{equation}
\end{lemma}

\begin{proof}
For $j=0,\ldots,n-1$, set
\begin{equation}\label{eq:epsilon-choice}
  \eps_j=\frac{\eta}{2nR^{n-j}}.
\end{equation}
Starting with $Z_n=N$, construct the lower strata in the order
$$
  Z_{n-1},Z_{n-2},\ldots,Z_0.
$$
Suppose that $Z_{j+1}$ has been chosen.  If it is empty, set all remaining lower strata equal to the empty set and stop.  If every connected component of $Z_{j+1}$ is contained in an $R$-ball, take $Z_j=\varnothing$, set all lower strata equal to the empty set, and stop; this is admissible by \cref{conv:empty} and has minimal area zero.  In the remaining case, the empty competitor is not admissible, while \cref{lem:separators-exist} gives $I_R(Z_{j+1})<\infty$.  We may therefore choose a nonempty $Z_j\subset Z_{j+1}$ with $j$-area within $\eps_j$ of that infimum.  This is the approximate-minimization scheme used in the proof of \cite[Lemma~7]{AlpertSeparating}.

Fix, after the filtration has been chosen, a subset $A\subset Z_0$ satisfying \eqref{eq:cluster-ball-condition}.  If $A=\varnothing$, then the left-hand side and the neighborhood volume vanish, so \eqref{eq:cluster-zero} reduces to $0\le\eta$.  Assume $A\ne\varnothing$.  Then $Z_0\ne\varnothing$, so the construction did not stop at an empty level; consequently every $Z_i$ is nonempty and, for $0\le i<n$, the separator $Z_i\subset Z_{i+1}$ was chosen within $\eps_i$ of $I_R(Z_{i+1})$.  For $i=0,\ldots,n$ put
$$
  F_i(u)=\Area_i\paren{Z_i\cap\cN_u^N(A)}.
$$
We prove by induction on $i$ the following statement, with the quantifier in $u$ included in the induction hypothesis: for every $u\in(0,t]$,
\begin{equation}\label{eq:induction-cluster}
  \#A\,\frac{u^i}{i!}
  \le
  F_i(u)
  +2\sum_{j=0}^{i-1}\eps_jR^{i-j}.
\end{equation}
For $i=0$, the sum is empty and
$$
  F_0(u)=\#\paren{Z_0\cap\cN_u^N(A)}\ge\#A
$$
for every $u>0$.

Assume \eqref{eq:induction-cluster} holds for some $i<n$ and every $u\in(0,t]$.  Fix $u\in(0,t]$, apply the induction hypothesis with $s$ in place of $u$, and integrate over $s\in(0,u)$:
\begin{align*}
  \#A\,\frac{u^{i+1}}{(i+1)!}
  &\le
  \int_0^u F_i(s)\,ds
  +2\sum_{j=0}^{i-1}\eps_jR^{i-j}u.
\end{align*}
Because $u\le t$, the compact-containment hypothesis also holds for the $u$-neighborhood.  Applying \cref{lem:cluster-slicing} with
$$
  P=Z_{i+1},\qquad Z=Z_i,\qquad a=0,\qquad b=u
$$
gives
$$
  \int_0^uF_i(s)\,ds
  \le F_{i+1}(u)+2\eps_iR.
$$
Since $u<R$,
$$
  R^{i-j}u\le R^{i+1-j}.
$$
Combining the last three displays proves \eqref{eq:induction-cluster} for $i+1$ and every $u\in(0,t]$.

Taking $i=n$ and $u=t$, we have $F_n(t)=\Vol\bigl(\cN_t^N(A)\bigr)$.  By \eqref{eq:epsilon-choice},
$$
  2\sum_{j=0}^{n-1}\eps_jR^{n-j}
  =2\sum_{j=0}^{n-1}\frac{\eta}{2n}=\eta.
$$
Thus \eqref{eq:induction-cluster} is exactly \eqref{eq:cluster-zero}.
\end{proof}

\begin{remark}[Simultaneity]\label{rem:simultaneity}
The strata depend only on the prescribed errors $\eps_j$.  Once the filtration is fixed, \cref{lem:cluster-slicing} can be applied separately to every subset $A\subset Z_0$.  No genericity choice made for one cluster is used in the construction of the filtration or in the proof for another cluster.  Since $Z_0$ is finite, there is also no set-theoretic issue in the phrase ``every subset.''
\end{remark}

\section{A local padded partition}\label{sec:padded}

We next prove the finite-metric partition lemma that converts local cardinality control into disjoint cluster neighborhoods.  We write
$$
  B_X(x,r)=\set{y\in X:d_X(x,y)<r}.
$$

\begin{lemma}[Local padded partition]\label{lem:padded}
Let $(X,d_X)$ be a finite metric space.  Suppose that for some $\Delta>0$ and integer $D\ge2$,
\begin{equation}\label{eq:local-D}
  \#B_X(x,\Delta)\le D
  \qquad\text{for every }x\in X.
\end{equation}
Set
$$
  L=\max\{1,\log D\},
  \qquad
  s=\frac{\Delta}{16L}.
$$
Then there is a partition $\cP$ of $X$ such that
\begin{enumerate}[label=(\roman*),leftmargin=2.5em]
\item every cluster $C\in\cP$ has $\diam C\le\Delta$;
\item at least $\#X/2$ points $x\in X$ are $s$-padded, meaning
$$
  B_X(x,s)\subset \cP(x),
$$
where $\cP(x)$ is the cluster containing $x$.
\end{enumerate}
\end{lemma}

\begin{proof}
Choose independently a random radius
$$
  \rho\sim\operatorname{Unif}\brak{\Delta/4,\Delta/2}
$$
and a uniformly random ordering of $X$.  Assign each point $z\in X$ to the first point $y$ in the ordering satisfying $d_X(y,z)<\rho$.  Such a point exists because $y=z$ is eligible.  Points assigned to the same center form a cluster.  If $z,z'$ belong to one cluster, then
$$
  d_X(z,z')<2\rho\le\Delta,
$$
so every cluster has diameter at most $\Delta$.

Fix $x\in X$ and define
$$
  m_x(u)=\#B_X(x,u).
$$
We say that $B_X(x,s)$ is \emph{cut} when it meets more than one cluster of the random partition.  Conditional on $\rho$, let $y_0$ be the first point in the random ordering belonging to $B_X(x,\rho+s)$.  If
$$
  y_0\in B_X(x,\rho-s),
$$
then every $z\in B_X(x,s)$ is assigned to $y_0$.  Indeed,
$$
  d_X(y_0,z)<(\rho-s)+s=\rho,
$$
and any center capable of capturing $z$ lies in $B_X(x,\rho+s)$, so none precedes $y_0$.  Therefore
\begin{equation}\label{eq:cut-conditional}
  \mathbb P\bigl[B_X(x,s)\text{ is cut}\mid\rho\bigr]
  \le
  1-\frac{m_x(\rho-s)}{m_x(\rho+s)}.
\end{equation}
Because $1-u/v\le\log(v/u)$ for $0<u\le v$, averaging \eqref{eq:cut-conditional} gives
\begin{align}
  \mathbb P\bigl[B_X(x,s)\text{ is cut}\bigr]
  &\le
  \frac4\Delta
  \int_{\Delta/4}^{\Delta/2}
  \bigl(\log m_x(\rho+s)-\log m_x(\rho-s)\bigr)\,d\rho.
  \label{eq:cut-integral}
\end{align}
Since $L\ge1$, one has $s\le\Delta/16$.  Hence, throughout the integration range,
$$
  0<\rho-s<\rho+s<\Delta,
$$
so $m_x(\rho-s)\ge1$ and \eqref{eq:local-D} gives
$$
  0\le\log m_x(\rho\pm s)\le\log D.
$$
Moreover,
$$
  \frac{\Delta}{4}+s\le\frac{\Delta}{2}-s,
$$
so the translated intervals overlap and their common part cancels.  Writing $f(u)=\log m_x(u)$, translation of the two integrals in \eqref{eq:cut-integral} therefore gives the exact identity
\begin{align*}
  \int_{\Delta/4}^{\Delta/2}
  \bigl(f(\rho+s)-f(\rho-s)\bigr)\,d\rho
  & =
  \int_{\Delta/2-s}^{\Delta/2+s}f(u)\,du
  -\int_{\Delta/4-s}^{\Delta/4+s}f(u)\,du\\
  &\le 2s\log D.
\end{align*}
Consequently,
$$
  \mathbb P\bigl[B_X(x,s)\text{ is cut}\bigr]
  \le \frac{8s\log D}{\Delta}\le\frac12.
$$
Thus each $x$ is padded with probability at least $1/2$.  The expected number of padded points is at least $\#X/2$, so some realization has at least that many.
\end{proof}

\section{The finite-scale logarithmic inequality}\label{sec:finite-scale}

We now combine the preceding ingredients.

\begin{proof}[Proof of \cref{thm:finite-scale-intro}]
Set
\begin{equation}\label{eq:scales}
\begin{aligned}
  D&=\max\set{2,\left\lceil
  n!8^n\paren{1+\frac{V_R(N)}{R^n}}
  \right\rceil},\\
  L&=\max\{1,\log D\},
  \qquad
  \tau=\frac{R}{512L}.
\end{aligned}
\end{equation}
Choose
\begin{equation}\label{eq:eta-main}
  \eta=\frac{\tau^n}{4n!}
\end{equation}
and let
$$
  N=Z_n\supset\cdots\supset Z_0
$$
be the filtration supplied by \cref{lem:cluster-zero}.  Put
$$
  X=Z_0,
  \qquad
  m=\#X,
$$
Equip $X$ with the restriction of $d_N$.

If $m=0$, \cref{prop:zero-stratum,lem:ball-pi1} give $\simp{N}=0$, and there is nothing to prove.  Assume $m>0$.

\smallskip
\noindent\emph{Local cardinality of $X$.}
For $x\in X$, let
$$
  A_x=X\cap B_N\paren{x,\frac R8}.
$$
Every point of $\overline{\cN_{R/8}^N(A_x)}$ has $d_N$-distance at most $R/4$ from $x$.  Hence
$$
  \overline{\cN_{R/8}^N(A_x)}
  \subset\overline{B_N\paren{x,\frac R4}}
  \subset B_N(x,R).
$$
Applying \cref{lem:cluster-zero} gives
$$
  \#A_x\,\frac{(R/8)^n}{n!}
  \le \Vol\bigl(\cN_{R/8}^N(A_x)\bigr)+\eta
  \le V_R(N)+\eta.
$$
Since $\tau<R$ and hence $\eta<R^n$, we obtain
\begin{equation}\label{eq:local-cardinality}
  \#B_X\paren{x,\frac R8}
  \le
  n!\frac{V_R(N)+R^n}{(R/8)^n}
  =n!8^n\paren{1+\frac{V_R(N)}{R^n}}
  \le D.
\end{equation}

Apply \cref{lem:padded} to $X$ at scale $R/8$.  Since the padding radius supplied by that lemma is
$$
  \frac{R/8}{16L}=\frac{R}{128L}=4\tau,
$$
we may choose a partition $\cP$ into clusters of diameter at most $R/8$ with at least $m/2$ points that are $4\tau$-padded.  For each cluster $C\in\cP$, define its padded core
\begin{equation}\label{eq:padded-core}
  A_C=\set{x\in C:B_X(x,4\tau)\subset C}.
\end{equation}
Then
\begin{equation}\label{eq:half-core}
  \sum_{C\in\cP}\#A_C\ge\frac m2.
\end{equation}
If $C\ne C'$ and $x\in A_C$, $y\in A_{C'}$, then $d_X(x,y)\ge4\tau$; otherwise $y\in B_X(x,4\tau)\subset C$.  Since $d_X$ is the restriction of $d_N$ and $2\tau<4\tau$, the ambient neighborhoods
$$
  \cN_\tau^N(A_C)
$$
of the nonempty cores are pairwise disjoint.

Choose $p_C\in A_C$ whenever $A_C\ne\varnothing$.  Because $\diam A_C\le R/8$,
$$
  \overline{\cN_\tau^N(A_C)}
  \subset
  \overline{B_N\paren{p_C,\frac R8+\tau}}
  \subset B_N(p_C,R).
$$
Thus \cref{lem:cluster-zero} applies to every core:
\begin{equation}\label{eq:core-estimate}
  \#A_C\,\frac{\tau^n}{n!}
  \le
  \Vol\bigl(\cN_\tau^N(A_C)\bigr)+\eta.
\end{equation}
Summing \eqref{eq:core-estimate}, using disjointness, and observing that the number of nonempty cores is at most $m$, we get
$$
  \frac{\tau^n}{n!}\sum_C\#A_C
  \le \Vol(N)+m\eta.
$$
By \eqref{eq:half-core} and \eqref{eq:eta-main},
$$
  \frac{m\tau^n}{2n!}
  \le \Vol(N)+\frac{m\tau^n}{4n!},
$$
and hence
\begin{equation}\label{eq:m-bound}
  m\le4n!\frac{\Vol(N)}{\tau^n}
  =4n!\,512^nL^n\frac{\Vol(N)}{R^n}.
\end{equation}

Since $R<\tfrac12\sys(N)$, \cref{lem:ball-pi1,prop:zero-stratum} imply
$$
  \simp{N}\le2^nm.
$$
Together with \eqref{eq:m-bound}, this gives
\begin{equation}\label{eq:pre-log}
  \simp{N}
  \le
  2^{n+2}n!\,512^nL^n\frac{\Vol(N)}{R^n}.
\end{equation}

It remains to compare $L$ with the logarithm in \cref{thm:finite-scale-intro}.  From the definition of $D$,
$$
  D\le2n!8^n\paren{1+\frac{V_R(N)}{R^n}}.
$$
Also
$$
  \log\paren{e+\frac{V_R(N)}{R^n}}\ge1
  \quad\text{and}\quad
  \log\paren{1+\frac{V_R(N)}{R^n}}
  \le\log\paren{e+\frac{V_R(N)}{R^n}},
$$
so
\begin{align*}
  L
  &\le 1+\log D\\
  &\le 1+\log(2n!8^n)
  +\log\paren{1+\frac{V_R(N)}{R^n}}\\
  &\le \brak{2+\log(2n!8^n)}
  \log\paren{e+\frac{V_R(N)}{R^n}}.
\end{align*}
Substituting this in \eqref{eq:pre-log} proves \cref{thm:finite-scale-intro}.
\end{proof}


\section{Finite covers with large metric systole}\label{sec:large-systole}

We record the covering lemma in the exact metric form needed below.

\begin{lemma}[Large-systole cover]\label{lem:large-systole}
Let $(M,g)$ be a closed connected Riemannian manifold with residually finite fundamental group.  For every $L>0$, there is a finite-sheeted normal cover $p:N\to M$ such that
$$
  \sys(N,p^*g)>L.
$$
\end{lemma}

\begin{proof}
Let $\Gamma=\pi_1(M)$ act by deck transformations on the universal cover $(\widetilde M,\widetilde g)$.  Fix $\widetilde x\in\widetilde M$, put
$$
  D=\diam_g(M)+1,
  \qquad
  L'=L+1,
$$
and consider
$$
  F=\set{\gamma\in\Gamma\setminus\{1\}:
  d_{\widetilde g}(\widetilde x,\gamma\widetilde x)\le L'+2D}.
$$
Properness of the deck action makes $F$ finite.  Residual finiteness supplies, for each $\gamma\in F$, a finite-index normal subgroup not containing $\gamma$.  Intersecting these finitely many subgroups gives a finite-index normal subgroup $H\triangleleft\Gamma$ with
$$
  H\cap F=\varnothing.
$$
Let $N=\widetilde M/H$, and let $p:N\to M$ be the induced finite normal covering.

Suppose, contrary to the desired conclusion, that $\sys(N,p^*g)\le L$.  By the definition of the infimum, $N$ then contains a non-null-homotopic loop of length less than $L'=L+1$.  Lifting this loop gives $h\in H\setminus\{1\}$ and $y\in\widetilde M$ with
$$
  d_{\widetilde g}(y,hy)<L'.
$$
Choose $k\in\Gamma$ with $d_{\widetilde g}(\widetilde x,ky)<D$; such a $k$ is obtained by lifting a path of length at most $\diam_g(M)$ between the projections of $y$ and $\widetilde x$.  Normality gives $khk^{-1}\in H$, and $khk^{-1}\ne1$.  Moreover,
$$
  d_{\widetilde g}(khy,khk^{-1}\widetilde x)=d_{\widetilde g}(ky,\widetilde x)<D,
$$
so
\begin{align*}
  d_{\widetilde g}(\widetilde x,khk^{-1}\widetilde x)
  &\le d_{\widetilde g}(\widetilde x,ky)
      +d_{\widetilde g}(ky,khy)
      +d_{\widetilde g}(khy,khk^{-1}\widetilde x)\\
  &<D+L'+D=2D+L'.
\end{align*}
Thus $khk^{-1}\in H\cap F$, a contradiction.  Therefore $\sys(N,p^*g)>L$.
\end{proof}

\begin{lemma}[Balls lift below half the systole]\label{lem:balls-lift}
If a Riemannian covering $\pi:\widetilde N\to N$ satisfies $\sys(N)>2r$, then for every $x\in N$ and every lift $\widetilde x\in\pi^{-1}(x)$, the restriction
$$
  \pi:B_{\widetilde N}(\widetilde x,r)\longrightarrow B_N(x,r)
$$
is a Riemannian isometry.  In particular,
$$
  V_r(N)=V_r(\widetilde N).
$$
\end{lemma}

\begin{proof}
Let $y\in B_N(x,r)$.  Since $d_N(x,y)<r$, the definition of the Riemannian distance gives a piecewise smooth path from $x$ to $y$ of length strictly less than $r$.  Lift this path starting at $\widetilde x$.  Its endpoint lies in $B_{\widetilde N}(\widetilde x,r)$, so the displayed restriction is surjective.

Suppose that distinct points $\widetilde y,\widetilde z\in B_{\widetilde N}(\widetilde x,r)$ have the same image.  Choose piecewise smooth paths from $\widetilde y$ to $\widetilde x$ and from $\widetilde x$ to $\widetilde z$, each of length strictly less than $r$, and let $\alpha$ be their concatenation.  Then
$$
  \operatorname{length}(\alpha)<2r,
$$
and $\pi\circ\alpha$ is a loop in $N$.  If this loop were null-homotopic, the homotopy-lifting property would imply that its lift starting at $\widetilde y$ is closed.  But $\alpha$ is that lift and ends at $\widetilde z\ne\widetilde y$.  Hence $\pi\circ\alpha$ is non-null-homotopic.  Since a Riemannian covering preserves the length of paths, this contradicts $\sys(N)>2r$.

Thus the restriction is injective.  It is a bijective local isometry, and therefore a Riemannian isometry.  Every center of either manifold is covered by this argument, so the corresponding ball volumes agree center by center.  Taking suprema proves $V_r(N)=V_r(\widetilde N)$.
\end{proof}

\section{Universal-cover ball growth}\label{sec:universal-cover}

We now prove the conjectural comparison.

\begin{proof}[Proof of \cref{thm:universal-cover-intro}]
The proportionality principle for closed hyperbolic manifolds implies
$$
  \simp{M}>0;
$$
indeed, $\simp{M}$ is a positive dimension-dependent multiple of the hyperbolic volume of $M$ \cite{Gromov82,LoehSurvey}.  Because $M$ admits a closed oriented hyperbolic metric, its hyperbolic holonomy identifies $\Gamma=\pi_1(M)$ with a finitely generated subgroup of
$$
  \mathrm{SO}^+(n,1)\subset\mathrm{GL}_{n+1}(\R).
$$
Mal'cev's theorem therefore implies that $\Gamma$ is residually finite \cite{Malcev}; see also \cite{Nica}.

Fix $r\ge1$ and suppose, toward a contradiction, that
\begin{equation}\label{eq:bad-radius}
  V_r(\widetilde M,\widetilde g)<V_r(\Hh^n).
\end{equation}
By \cref{lem:large-systole}, there is a finite normal cover $p:N\to M$ of degree $d$ such that, for the pulled-back metric,
$$
  \sys(N,p^*g)>2r.
$$
Simplicial volume and Riemannian volume are multiplicative under finite covers \cite{Gromov82,LoehSurvey}, so
\begin{equation}\label{eq:cover-multiplicity}
  \simp{N}=d\simp{M},
  \qquad
  \Vol_{p^*g}(N)=d\Vol_g(M).
\end{equation}
The universal cover of $N$ is canonically the same manifold $\widetilde M$ as the universal cover of $M$, and the lift of $p^*g$ is $\widetilde g$.  Applying \cref{lem:balls-lift} to the universal covering $\widetilde M\to N$ gives
\begin{equation}\label{eq:cover-ball}
  V_r(N,p^*g)=V_r(\widetilde M,\widetilde g).
\end{equation}
Since $r<\tfrac12\sys(N,p^*g)$, \cref{thm:finite-scale-intro} applied to $(N,p^*g)$ at scale $r$ yields
\begin{equation}\label{eq:apply-finite-scale}
  d\simp{M}
  \le
  C_n\frac{d\Vol_g(M)}{r^n}
  \brak{\log\paren{e+\frac{V_r(\widetilde M,\widetilde g)}{r^n}}}^{n}.
\end{equation}

Let $\omega_{n-1}$ be the Euclidean area of the unit $(n-1)$-sphere.  Hyperbolic polar coordinates give
$$
  V_r(\Hh^n)=\omega_{n-1}\int_0^r\sinh^{n-1}(t)\,dt.
$$
Since $\sinh t\le e^t/2$ for $t\ge0$,
\begin{equation}\label{eq:hyperbolic-upper}
  V_r(\Hh^n)
  \le
  b_ne^{(n-1)r},
  \qquad
  b_n=\frac{\omega_{n-1}}{2^{n-1}(n-1)}.
\end{equation}
For $r\ge1$, \eqref{eq:bad-radius} and \eqref{eq:hyperbolic-upper} imply
\begin{align}
  \log\paren{e+\frac{V_r(\widetilde M,\widetilde g)}{r^n}}
  &<\log\paren{e+\frac{V_r(\Hh^n)}{r^n}}\notag\\
  &\le\log\paren{e+b_ne^{(n-1)r}}\notag\\
  &\le K_nr,\label{eq:log-hyp}
\end{align}
where one may take
$$
  K_n=n-1+\log(e+b_n).
$$
Indeed,
$$
  e+b_ne^{(n-1)r}
  \le(e+b_n)e^{(n-1)r},
$$
and $r\ge1$.  Substituting \eqref{eq:log-hyp} in \eqref{eq:apply-finite-scale} and cancelling $d$ gives
$$
  \simp{M}<C_nK_n^n\Vol_g(M).
$$
Since $\simp{M}>0$, every bad radius forces
\begin{equation}\label{eq:ratio-lower}
  \frac{\Vol_g(M)}{\simp{M}}>\frac{1}{C_nK_n^n}.
\end{equation}
Choose
\begin{equation}\label{eq:delta-explicit}
  \delta_n=\frac{1}{C_nK_n^n}.
\end{equation}
Then \eqref{eq:conjecture-four-hypothesis} contradicts \eqref{eq:ratio-lower}.  Since the argument applies to every fixed $r\ge1$, no bad radius exists, proving \eqref{eq:conjecture-four-conclusion}.
\end{proof}

\begin{remark}
The finite cover is allowed to depend on the hypothetical bad radius $r$.  The ratio $\Vol_g(M)/\simp{M}$ is invariant under finite covers, which is exactly why the argument yields a single constant, depending only on the dimension, that is valid simultaneously for all radii.
\end{remark}

\begingroup
\renewcommand{\bibliofont}{\small}
\sloppy
\setlength{\emergencystretch}{2em}

\endgroup

\end{document}